\documentclass[11pt,english]{article}
\usepackage{physics}
\usepackage{mathtools}
\usepackage{mathrsfs}
\usepackage{graphicx}
\usepackage{graphics}
\usepackage{setspace}
	\usepackage{color}
\usepackage{amsmath,amssymb,amsthm}
\usepackage{epsfig}
\usepackage{babel}
\usepackage{tikz}
\usepackage{tikz-cd}
\usetikzlibrary{arrows,shapes,automata,backgrounds,petri,positioning,scopes,matrix, calc, plothandlers,plotmarks,patterns}
\usepackage{enumerate}
\usepackage{float}
\usepackage[utf8]{inputenc}
\usepackage[T1]{fontenc}
\usepackage[document]{ragged2e}
 \usepackage{verbatim}
\usepackage{tikz-3dplot}
\usepackage{pgfplots}
\usepackage{tkz-graph}
\GraphInit[vstyle = Normal]
\tikzset{
  LabelStyle/.style = {minimum width = 2em, 
                        text = red, font = \bfseries },
  VertexStyle/.append style = { inner sep=2pt,
                                font = \Large\bfseries, fill},
  EdgeStyle/.append style = {->, bend left} }

\newtheorem{thm}{Theorem}[section]
\numberwithin{equation}{section} 
\numberwithin{figure}{thm} 
\theoremstyle{plain}
\newtheorem*{thm*}{Theorem}
\theoremstyle{definition}
\theoremstyle{plain}
\newtheorem{thm_A}{Theorem}
\newtheorem*{defn*}{Definition}

\theoremstyle{plain}

\theoremstyle{plain} 

\theoremstyle{plain}
\theoremstyle{definition}
\newtheorem{ex}[thm]{Example}
\theoremstyle{remark}

\theoremstyle{plain}

\theoremstyle{plain}

\theoremstyle{plain}
\newtheorem{lem}[thm]{Lemma}
\newtheorem*{lem*}{Lemma} 
\theoremstyle{definition}
\newtheorem{defn}[thm]{Definition}

\newtheorem*{acknowledgment*}{Addentum}

\theoremstyle{plain}
\newtheorem*{ex*}{Example}
\theoremstyle{plain}

\AtBeginDocument{
  
}
\begin{document}
\pgfdeclarelayer{background}
\pgfsetlayers{background,main}
\title{About an extension of the Davenport-Rado result to the  Herzog-Sch\"onheim conjecture for free groups}
\author{Fabienne Chouraqui}

\date{}

\maketitle
\begin{abstract}
Let $G$ be a group and $H_1$,...,$H_s$ be subgroups of $G$ of  indices $d_1$,...,$d_s$ respectively. In 1974, M. Herzog and J. Sch\"onheim conjectured that if $\{H_i\alpha_i\}_{i=1}^{i=s}$,  $\alpha_i\in G$, is a coset partition of $G$, then $d_1$,..,$d_s$ cannot be distinct. We consider the  Herzog-Sch\"onheim conjecture for free groups of finite rank and propose a new approach, based on an extension of the  Davenport-Rado result for $G=\mathbb{Z}$.
\end{abstract}
\maketitle
\section{Introduction}
Let $G$ be a group and $H_1$,...,$H_s$ be subgroups of $G$.  If there exist  $\alpha_i\in G$ such that $G= \bigcup\limits_{i=1}^{i=s}H_i\alpha_i$, and the sets  $H_i\alpha_i$, $1 \leq i \leq s$,  are pairwise disjoint, then  $\{H_i\alpha_i\}_{i=1}^{i=s}$ is \emph{a coset partition of $G$}  (or a \emph{disjoint cover of $G$}). In this case,    all the subgroups  $H_1$,...,$H_s$ can be assumed to be of  finite index in  $G$ \cite{newman,korec}. We denote by $d_1$,...,$d_s$ the indices of $H_1$,...,$H_s$ respectively. The coset partition $\{H_i\alpha_i\}_{i=1}^{i=s}$ has  \emph{multiplicity} if $d_i=d_j$ for some $i \neq j$. 

\setlength\parindent{10pt} In 1974, M. Herzog and J. Sch\"onheim  conjectured that any coset partition of any group  $G$ has multiplicity. In the 1980's, in a series of papers,  M.A. Berger, A. Felzenbaum and A.S. Fraenkel studied  the Herzog-Sch\"onheim conjecture \cite{berger1, berger2,berger3} and in \cite{berger4} they proved the conjecture is true for the pyramidal groups, a subclass of the finite solvable groups. Coset partitions of finite groups with additional assumptions on the subgroups of the partition have been extensively studied. We refer to \cite{brodie,tomkinson1, tomkinson2,sun}. In \cite{schnabel}, the authors very recently proved that the conjecture is true for all groups of order less than $1440$. 

 In \cite{chou_hs}, we consider  free groups of finite rank and develop a new combinatorial approach to the problem, based on the machinery of covering spaces. The  fundamental group of the  bouquet with $n\geq 1$ leaves (or the wedge sum of $n$ circles),  $X$,  is $F_n$, the  free group of finite rank $n$. As  $X$ is a ``good'' space (connected, locally path connected and semilocally $1$-connected), $X$ has a  universal covering  which can be identified with the Cayley graph of $F_n$,  an infinite simplicial tree. Furthermore, there exists a  one-to-one correspondence between the subgroups of $F_n$ and the covering spaces (together with a chosen point) of $X$.  

 For any  subgroup $H$ of $F_n$ of finite index $d$, there exists  a $d$-sheeted covering space  $(\tilde{X}_H,p)$  with a fixed basepoint. The underlying graph of $(\tilde{X}_H,p)$ is a directed labelled graph with $d$ vertices. We call it  \emph{the  Schreier graph of $H$} and denote it by $\tilde{X}_{H}$. It can be seen also as  a finite complete bi-deterministic automaton; fixing the start and the end state at the basepoint, it recognises the set of elements in $H$.  It is   called \emph{the Schreier coset diagram  for $F_n$ relative to the subgroup  $H$} \cite[p.107]{stilwell} or  \emph{the Schreier automaton for $F_n$ relative to the subgroup $H$} \cite[p.102]{sims}. The $d$ vertices (or states) correspond to the $d$ right cosets of $H$,  each edge (or transition) $Hg \xrightarrow{a}Hga$, $g \in F_n$, $a$ a generator of $F_n$,  describes the right action of $a$ on  $Hg$. If we fix the start state at  $H$, the basepoint,  and the end state at another vertex  $Hg$, where $g$  denotes the label of some path from the start state to the end state, then this automaton recognises the set of elements in $Hg$ and we call it  \emph{the  Schreier automaton  of $Hg$} and denote it by $\tilde{X}_{Hg}$.\\

In general, for any  automaton $M$, with alphabet $\Sigma$, and $d$ states,  there exists a  square matrix $A$ of order $d\times d$, with $a_{ij}$ equal to the number of directed edges from  vertex $i$ to vertex $j$, $1\leq i,j\leq d$. This matrix is   non-negative and it is  called  \emph{the transition matrix} \cite{epstein-zwik}. If for every $1\leq i,j\leq d$, there exists $m \in \mathbb{Z}^+$ such that $(A^m)_{ij}>0$, the matrix is \emph{irreducible}. For an irreducible non-negative matrix  $A$,  \emph{the period of $A$} is  the gcd of all $m \in \mathbb{Z}^+$ such that $(A^m)_{ii} >0$ (for any $i$).  If  $i$ and $j$ denote respectively the start and end  states of $M$, then the number of  words of length $k$ (in the alphabet $\Sigma$) accepted by $M$ is $a_k=(A^k)_{ij}$. \emph{The generating function of $M$} is defined by  $p(z)=\sum\limits_{k=0}^{k=\infty}a_k\,z^k$.  It is a rational function: the fraction of two polynomials in $z$ with integer coefficients \cite{epstein-zwik}, \cite[p.575]{stanley}.\\

 The intuitive idea behind our  approach  in this paper is  as follows. Let $F_n=\langle \Sigma\rangle$, and $\Sigma^*$ the free monoid  generated by $\Sigma$. Let $\{H_i\alpha_i\}_{i=1}^{i=s}$ be a coset  partition of $F_n$  with $H_i<F_n$ of index $d_i>1$, $\alpha_i \in F_n$, $1 \leq i \leq s$. Let $\tilde{X}_{i}$ denote the  Schreier  automaton of $H_i\alpha_i$, with transition matrix $A_i$ and generating function  $p_i(z)$, $1 \leq i\leq s$. For each $\tilde{X}_{i}$,  $A_i$ is a non-negative irreducible matrix and $a_{i,k}$, $k \geq 0$, counts the  number of  words of length $k$ that belong to $H_i\alpha_i\cap \Sigma^*$.
 Since $F_n$ is the disjoint union of the sets  $\{H_i\alpha_i\}_{i=1}^{i=s}$, each element  in $\Sigma^*$ belongs to one and exactly one such set, so $n^k$, the number of  words of length $k$ in $\Sigma^*$, satisfies $n^k=\sum\limits_{i=1}^{i=s}a_{i,k}$, for every $k \geq 0$. So, $\sum\limits_{k=0}^{k=\infty}n^k\,z^k=\sum\limits_{i=1}^{i=s}p_i(z)$.  Using this kind of counting argument, we prove that there is a repetition of the maximal  period $h>1$ and in some cases we could prove there is a repetition of the index also.





 \begin{thm_A}\label{theo0}
 Let $F_n$ be the free group on $n \geq 1$ generators. Let $\{H_i\alpha_i\}_{i=1}^{i=s}$ be a coset  partition of $F_n$ with $H_i<F_n$ of index $d_i$, $\alpha_i \in F_n$, $1 \leq i \leq s$, and $1<d_1 \leq ...\leq d_s$.  Let $\tilde{X}_{i}$ denote the  Schreier  graph of $H_i$, with  transition matrix   $A_i$, and period $h_i \geq 1$, $1 \leq i\leq s$. Let $1 \leq k,m\leq s$.
\begin{enumerate}[(i)]
\item Assume $h_k>1$, where  $h_k=max\{h_i \mid 1 \leq i \leq s\}$.   Then there exists $j\neq k$ such that  $h_j=h_k$, that is there is a repetition of the maximal period. 
\item Let $h_{\ell}>1$, such that $h_{\ell}$ does not properly divide any other period $h_i$, $ 1 \leq i \leq s$.  Then there exists  $j\neq \ell$ such that  $h_j=h_{\ell}$.
\item  For every $h_i$, there exists  $j\neq i$ such that  either $h_i=h_j$ or $h_i\mid h_j$.
\end{enumerate}
  \end{thm_A} 

	If   $n=1$ in Theorem \ref{theo0}, $\{H_i\alpha_i\}_{i=1}^{i=s}$  is a coset  partition of $\mathbb{Z}$. A coset partition of $\mathbb{Z}$ is $\{d_i\mathbb{Z} +r_i\}_{i=1}^{i=s}$, $r_i \in \mathbb{Z}$,  with  each $d_i\mathbb{Z} +r_i$ the residue class of $r_i$ modulo $d_i$. These coset partitions of $\mathbb{Z}$ were first introduced by P. Erd\H{o}s \cite{erdos1} and he conjectured that if $\{d_i\mathbb{Z} +r_i\}_{i=1}^{i=s}$, $r_i \in \mathbb{Z}$, is a  coset partition of $\mathbb{Z}$, then  the largest index $d_s$ appears at least twice.  Erd\H{o}s' conjecture was proved  by H. Davenport and R.Rado,  and independently by  L. Mirsky and  D. Newman (not published) using analysis of complex function \cite{erdos2,newman,znam}. Furthermore, it was proved that  the largest index $d_s$ appears at least $p$ times,  where $p$ is the smallest prime dividing $d_s$ \cite{newman,znam,sun2}, that each index $d_i$  divides another index $d_j$, $j\neq i$, and  that each index $d_k$ that does not properly divide any other index  appears at least twice \cite{znam}. We refer to \cite{ginosar} for a recent proof.\\
	
	With Theorem \ref{theo0}, we recover the  Davenport-Rado result (or Mirsky-Newman result) for the Erd\H{o}s' conjecture and some of its consequences. Indeed, for every index $d$, the Schreier graph of $d\mathbb{Z}$ has a transition matrix with period equal to $d$, so a repetition of the period is equivalent to a repetition of the index. For the  unique subgroup $H$ of $\mathbb{Z}$ of  index $d$,  its Schreier graph  $\tilde{X}_{H}$ is a closed directed path of length $d$ (with each edge labelled $1$). So, its  transition matrix   $A$ is the permutation matrix corresponding to the $d-$cycle $(1,2,...,d)$, and it has   period $d$.  In particular, the period of $A_s$ is $d_s$, and  there exists $j\neq s$ such that $d_j=d_s$.  Also, if the period (index) $d_k$ of $A_k$ does not properly divide any other period (index),  then  there exists  $j\neq k$ such that $d_j=d_k$. \\
	
	For the free groups in general, we prove that in some cases, the repetition of the period implies the repetition of the index. More precisely:

 \begin{thm_A}\label{theo1}
 Let $F_n$ be the free group on $n \geq 1$ generators. Let $\{H_i\alpha_i\}_{i=1}^{i=s}$ be a coset  partition of $F_n$ with $H_i<F_n$ of index $d_i$, $\alpha_i \in F_n$, $1 \leq i \leq s$, and $1<d_1 \leq ...\leq d_s$.  Let $\tilde{X}_{i}$ denote the  Schreier  graph of $H_i$, with  transition matrix   $A_i$, and period $h_i \geq 1$, $1 \leq i\leq s$. 
\begin{enumerate}[(i)]
\item If the period of $A_s$ is $d_s$, then $\{H_i\alpha_i\}_{i=1}^{i=s}$ has multiplicity and there exists  $j\neq s$ such that $d_j=d_s$. 
 
\item Let $h>1$, where  $h$ is either equal to $max\{h_i \mid 1 \leq i \leq s\}$ or such that $h$ does not properly divide any other period. Let $J=\{j \,\mid\, 1 \leq j\leq s,\,h_j=h\}$. Let $k \in J$ such that $d_k=max\{d_j \mid j \in J\}$.  If the period of $A_k$ is $d_k$, then   $\{H_i\alpha_i\}_{i=1}^{i=s}$ has multiplicity with  $j \in J$,  $j\neq k$ such that $d_j=d_k$. 
\end{enumerate}
  \end{thm_A} 

    The paper is organized as follows.  In  Section $1$, we give some preliminaries on automatons and their growth functions. In Section $2$, we prove the main result.    We also  refer to \cite{chou_hs} for more preliminaries and examples: Section 2, for free groups and covering spaces and   Section 3.1, for graphs.    
                    
 \section{Premilinaries on Automata}
\subsection{Automata and generating  function of their language}\label{subsec_automat}
We refer the reader to \cite[p.96]{sims}, \cite[p.7]{epstein}, \cite{pin,pin2}, \cite{epstein-zwik}.
A \emph{finite state automaton} is a quintuple $(S,\Sigma,\mu,Y,s_0)$, where $S$ is a finite set, called the \emph{state set}, $\Sigma$ is a finite set, called the \emph{alphabet}, $\mu:S\times \Sigma \rightarrow S$ is a function, called the \emph{transition function}, $Y$ is a (possibly empty) subset of $S$ called the \emph{accept (or end) states}, and $s_0$ is called the \emph{start state}.  It is a directed  graph with vertices the states and each transition $s \xrightarrow{a} s'$ between states $s$ and $s'$ is an edge with label $a \in \Sigma$. The \emph{label of a path $p$} of length $n$  is the product $a_1a_2..a_n$ of the labels of the edges of $p$.
The  finite state automaton $M=(S,\Sigma,\mu,Y,s_0)$ is \emph{deterministic} if there is only one initial state and each state is the source of exactly one arrow with any given label from  $\Sigma$. In a deterministic automaton, a path is determined by its starting point and its label \cite[p.105]{sims}. It is \emph{co-deterministic} if there is only one final state and each state is the target of exactly one arrow with any given label from  $\Sigma$. The  automaton $M=(S,\Sigma,\mu,Y,s_0)$ is \emph{bi-deterministic} if it is both deterministic and co-deterministic. An automaton $M$ is \emph{complete} if for each state $s\in S$ and for each $a \in \Sigma$, there is exactly one edge from $s$ labelled $a$.
\begin{defn}
Let $M=(S,\Sigma,\mu,Y,s_0)$ be  a finite state automaton. Let $\Sigma^*$ be the free monoid generated by $\Sigma$. Let  $\operatorname{Map}(S,S)$ be  the monoid consisting of all maps from $S$ to $S$. The map $\phi: \Sigma \rightarrow \operatorname{Map}(S,S) $ given by $\mu$ can be extended in a unique way to a monoid homomorphism $\phi: \Sigma^* \rightarrow \operatorname{Map}(S,S)$. The range of this map is a monoid called \emph{the transition monoid of $M$}, which is generated by $\{\phi(a)\mid a\in \Sigma\}$. An element $w \in \Sigma^*$ is \emph{accepted} by $M$ if the corresponding element of $\operatorname{Map}(S,S)$, $\phi(w)$,  takes $s_0$ to an element of the accept states set $Y$. The set $ L\subseteq \Sigma^*$  recognized by $M$ is called \emph{the language accepted by $M$}, denoted by $L(M)$.
\end{defn}
For any directed  graph with $d$ vertices or any finite state  automaton $M$, with alphabet $\Sigma$, and $d$ states,  there exists a  square matrix $A$ of order $d\times d$, with $a_{ij}$ equal to the number of directed edges from  vertex $i$ to vertex $j$, $1\leq i,j\leq d$. This matrix is   non-negative and it is  called  \emph{the transition matrix} (as in \cite{epstein-zwik}) or  \emph{the adjacency matrix} (as in \cite[p.575]{stanley}). For any $k \geq 1$, $(A^k)_{ij}$ is  equal to the number of directed paths of length $k$   from  vertex $i$ to vertex $j$. If for every $1\leq i,j\leq d$, there exists $m \in \mathbb{Z}^+$ such that $(A^m)_{ij}>0$, the matrix is \emph{irreducible}. For $A$  an irreducible non-negative matrix, \emph{the period of $A$} is  the gcd of all $m \in \mathbb{Z}^+$ such that $(A^m)_{ii} >0$ (for any $i$). If the period is $1$, A is \emph{aperiodic}.  \\

Let  $M$ be a bi-deterministic automaton  with alphabet $\Sigma$, $d$ states, start state $i$, accept state  $f$  and  transition matrix $A$. Let   $a_k=(A^k)_{if}$, the number of  words of length $k$ in the free monoid $\Sigma^*$,  accepted by $M$. The function  $p_{i  f}(z)=\sum\limits_{k=0}^{k=\infty}a_k\,z^k\,$ is called \emph{the generating function of $M$} \cite[p.574]{stanley}. 
\begin{thm}\cite[p.574]{stanley}\label{theo_genfn_stanley}
The generating function $p_{if}(z)$ is given by
\[p_{if}(z)=\frac{(-1)^{ij+f}det(I-zA:f,i)}{det(I-zA)}\]
where $(B : f, i)$ denotes the matrix obtained by removing the $f$th row and $i$th column of B, $det(I-zA)$ is the reciprocal polynomial of the characteristic polynomial of $A$.
Thus in particular $p_{if}(z)$ is a rational function whose degree is strictly less than the
algebraic multiplicity  of $0$ as an eigenvalue of $A$.
\end{thm}
\subsection{The Schreier automaton of a coset of a subgroup of $F_n$}

We now introduce the particular automata we are interested in, that is \emph{the Schreier coset diagram  for $F_n$ relative to the subgroup  $H$} \cite[p.107]{stilwell} or  \emph{the Schreier automaton for $F_n$ relative to the subgroup $H$} \cite[p.102]{sims}. 
\begin{defn}\label{def_Schreier-graph}
 Let $F_n=\langle\Sigma\rangle$ and  $\Sigma^*$ the free monoid generated by $\Sigma$.  Let $H<F_n$ of index $d$. Let $(\tilde{X}_H,p)$ be the covering  of the $n$-leaves bouquet with basepoint $\tilde{x}_1$ and  vertices  $\tilde{x}_1, \tilde{x}_2,...,\tilde{x}_{d}$. Let  $t_i \in \Sigma^*$ denote the label of a path from $\tilde{x}_1$ to $\tilde{x}_i$. Let $\mathscr{T}=\{1, t_i\mid 1 \leq i\leq d\}$.  Let $\tilde{X}_H$ be the Schreier coset diagram  for $F_n$ relative to the subgroup  $H$,  with $\tilde{x}_1$ representing the subgroup $H$ and the other vertices $\tilde{x}_2,...,\tilde{x}_{d}$ representing the cosets  $Ht_i$ accordingly.  We call $\tilde{X}_{H}$  \emph{the  Schreier graph  of $H$},  with this correspondence between the vertices  $\tilde{x}_1, \tilde{x}_2,...,\tilde{x}_{d}$ and the cosets  $Ht_i$ accordingly.  
   \end{defn}

  From its definition, $\tilde{X}_{H}$  is  a strongly-connected graph (i.e any two vertices are connected by a directed path), so its transition  matrix $A$ is non-negative and irreducible. As $\tilde{X}_{H}$  is a directed $n$-regular graph, the sum of the elements at each row and at each column of  $A$ is equal to $n$. So, from the Perron-Frobenius result for non-negative irreducible matrices, $n$ is the \emph{Perron-Frobenius eigenvalue} of $A$, that is the positive real eigenvalue with maximal absolute value. If $A$ has period $h\geq 1$, then  $A$ is similar to the matrix $Ae^{\frac{2\pi i}{h}}$, that is the set $\{\lambda e^{\frac{2\pi ik}{h}}\mid 0 \leq k\leq h-1\}$ is a set of eigenvalues of $A$, for each  eigenvalue $\lambda$ of $A$. In particular,  $\{n e^{\frac{2\pi ik}{h}}\mid 0 \leq k\leq h-1\}$ is a set of simple eigenvalues of $A$ \cite[Ch.16]{bellman}.
  \begin{thm}\cite[Thm. V7]{flajolet}\label{theo_genfn_flajolet}
  Let $A$ be non-negative and irreducible matrix of order $d\times d$. Let $P(z)=(I-zA)^{-1}$. Then all the entries $p_{if}(z)$ of $P(z)$ have the same radius of convergence $\frac{1}{\lambda_{PF}}$, where $\lambda_{PF}$ is the Perron-Frobenius eigenvalue of $A$.
  \end{thm}
\begin{defn}\label{def_Schreier-automaton}
 Let $F_n=\langle\Sigma\rangle$ and $\Sigma^*$ the free monoid generated by $\Sigma$.  Let $H<F_n$ of index $d$. Let  $\tilde{X}_{H}$  be the  Schreier graph  of $H$. Using the notation from Defn. \ref{def_Schreier-graph}, let $\tilde{x}_1$ be the start state and $\tilde{x}_f$ be the end state for some $1 \leq f \leq d$. We call  the automaton obtained  \emph{the  Schreier automaton   of $Ht_f$}   and denote it by  $\tilde{X}_{Ht_f}$. The language accepted by $\tilde{X}_{Ht_f}$ is  the set of elements in $\Sigma^*$  that belong to $Ht_f$. We call the elements in  $\Sigma^*\cap Ht_f$,  \emph{the positive words in $Ht_f$}. The identity may belong to  this set.
   \end{defn} 
In contrast with our approach in \cite{chou_hs}, as we are interested here in counting positive words of a given length, we do  not add the inverses of the generators from $\Sigma$ to the alphabet.
\begin{ex}\label{ex_automaton_index3}
 Let  $\Sigma=\{a,b\}$ be an alphabet. Let $\Sigma^*$ be the free monoid generated by $\Sigma$. Let  $F_2=\langle a, b \rangle$. 
Let $H= \langle a^4,b^4,ab^{-1},a^2b^{-2},a^3b^{-3} \rangle$ be a subgroup of index $4$ in $F_2$. Let $\tilde{X}_H$ be the Schreier graph of $H$: 
\begin{figure}[H] 
   \centering \scalebox{0.9}[0.8]{\begin{tikzpicture}
   \SetGraphUnit{4}
    \tikzset{VertexStyle/.append  style={fill}}
     \Vertex[L=$H$, x=-3,y=0]{A}
     \Vertex[L=$Ha$, x=0, y=0]{B}
   
   \Vertex[L=$Ha^{2}$, x=3, y=0]{C}
    \Vertex[L=$Ha^{3}$, x=6, y=0]{D}
   \Edge[label = a, labelstyle = above](A)(B)
    \Edge[label = a, labelstyle = above](B)(C)
     \Edge[label = a, labelstyle = above](C)(D)
    \Edge[label = b, labelstyle = below](D)(A)
 \tikzset{EdgeStyle/.style = {->}}    
  \Edge[label =b, labelstyle = below](A)(B)
 \Edge[label = b, labelstyle = below](B)(C)
   \Edge[label = b, labelstyle = below](C)(D)

\tikzset{EdgeStyle/.style = {->, bend right}} 
\Edge[label = a, labelstyle = above](D)(A)
  
 \end{tikzpicture}}

 \caption{The Schreier graph $\tilde{X}_H$  of $H= \langle a^4,b^4,ab^{-1},a^2b^{-2},a^3b^{-3} \rangle$.}
\label{fig_aut2}
\end{figure}
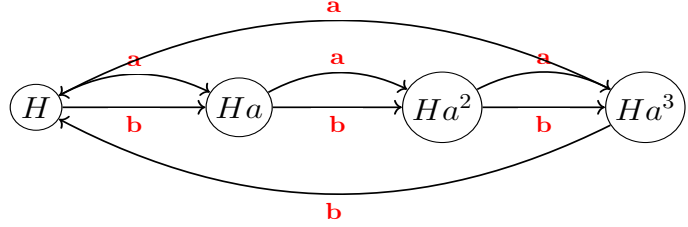The transition matrix of $\tilde{X}_H$ is 
$ \left( \begin{array}{cccc}
                       0 & 2 & 0 & 0 \\
                        0 & 0 & 2 & 0\\
                        0 &  0 & 0 & 2\\
						2 & 0  & 0 & 0\\

                       \end{array} \right)$
and its period is $4$. If $H$ is both the start and accept state, then the language accepted,  $L$,   is  the set of elements in $\Sigma^*$  that belong to $H$, that is the set of  positive words in $H$. The generating function is then $p(z)=\frac{1}{1-16z^4}$, with $p(0)=1$, since it contains $1$. If $H$ is the start state and $Ha$ the  accept state, then $L$ is  the set of positive words in the coset $Ha$, and $p(z)=\frac{2z}{1-16z^4}$. 
\end{ex}


\section{The generating functions of Schreier automata}
\subsection{Properties of the Schreier automaton}
Let $H < F_n$ of index $d$. We prove some properties of its Schreier automaton,   its  transition matrix  and its generating function.

\begin{lem}\label{lem_1-nz}
 Let $H < F_n$ of index $d$, with Schreier graph $\tilde{X}_H$ and transition matrix $A$ with period $h\geq1$. Let $\lambda$ be  a non-zero  eigenvalue of $A$ of algebraic multiplicity $n_{\lambda}$. Then   $\frac{1}{\lambda}$ is a  pole of $\frac{1}{det(I-zA)}$ of order  $n_{\lambda}$.  Moreover,   $\{\frac{1}{n} e^{\frac{2\pi im}{h}}\mid 0 \leq m\leq h-1\}$ is a set of simple poles of $\frac{1}{det(I-zA)}$ of minimal absolute value. 
\end{lem}
  \begin{proof}
For any eigenvalue $\lambda$ of $A$ with algebraic multiplicity $n_\lambda$, $1-\lambda z$ is an eigenvalue of   $I-zA$ with   same algebraic multiplicity  $n_\lambda$. And if $\lambda\neq 0$,  $\frac{1}{\lambda}$  is a  pole of $\frac{1}{det(I-zA)}$ of  order $n_\lambda$. From the Perron-Frobenius result for non-negative irreducible matrices, $n$ is the Perron-Frobenius eigenvalue of $A$. So,  $\frac{1}{n}$ is a simple pole of $\frac{1}{det(I-zA)}$. As $n$ is the eigenvalue of $A$ of maximal absolute value, $\frac{1}{n}$ is the pole of $\frac{1}{det(I-zA)}$ of minimal absolute value. The same holds for $\frac{1}{n} e^{\frac{2\pi im}{h}}$.
\end{proof}
\begin{lem}\label{lem_genfn=sum}
 Let $H < F_n$ of index $d$, with Schreier graph $\tilde{X}_H$ and transition matrix $A$. Let $p_{ij}(z)$ be the generating function of the Schreier automaton, with  $i$ and $j$  the start and end states respectively. Then for $\abs{z}<\frac{1}{n}$, and every $1 \leq i \leq d$, 
 \[\sum\limits_{j=1}^{j=d}p_{ij}(z)=\sum\limits_{j=1}^{j=d}\frac{(-1)^{i+j}det(I-zA:j,i)}{det(I-zA)}= \frac{1}{1-nz}\]
\end{lem}
  \begin{proof}
The number of positive words of length $k\geq 0$ in $F_n$ is $n^k$, so the generating  function of $F_n$ is $\sum\limits_{k=0}^{k=\infty}n^kz^k=\frac{1}{1-nz}$,  for $z$ with  $\abs{z} <\frac{1}{n}$. As $F_n$ is the disjoint union of the $d$ cosets of $H$, the generating  function of $F_n$ is equal to the sum of the  generating functions corresponding to each coset of $H$.
\end{proof}

\begin{lem}\label{lem_asympto}
 Let $H < F_n$ of index $d$, with Schreier graph $\tilde{X}_H$. Let $p_{ij}(z)$ and $p_{kl}(z)$ be the generating functions of the Schreier automatons corresponding to  $i,j$ and $k,l$   respectively. Then,  $p_{ij}(z)$ and $p_{kl}(z)$ have the same radius of convergence  $\frac{1}{n}$, and  the same poles (of the same order).
\end{lem}

  \begin{proof}
  From Theorem \ref{theo_genfn_flajolet}, $p_{ij}(z)$ and $p_{kl}(z)$ have the same radius of convergence  $\frac{1}{n}$, since $n$ is the Perron-Frobenius eigenvalue of $A$. 
   Moreover, since $(I-zA)^{-1}=I+zA(I-zA)^{-1}$
   and $(I-zA)^{-1}=I+z(I-zA)^{-1}A$, and $A$ is irreducible, each  $p_{ij}(z)$  is positively linearly related to any other $p_{kl}(z)$.  Thus, the $p_{ij}(z)$ must all become infinite as soon as one of them does and at the same rate \cite[Ch.V]{flajolet}. Consequently,   $p_{ij}(z)$ and $p_{kl}(z)$  have the same poles (of the same order). 
\end{proof} 

 \begin{lem}\label{lem_set-poles-simple}
  Let $H < F_n$ of index $d$, with Schreier graph $\tilde{X}_H$ and transition matrix $A$ with period $h\geq 1$. Let $p_{ij}(z)$ be the generating function of the Schreier automaton, with  $i$ and $j$  the start and end states respectively. Then 
 for every $1 \leq i,j \leq d$, $\{\frac{1}{n} e^{\frac{2\pi im}{h}}\mid 0 \leq m\leq h-1\}$ is a set of simple poles of $p_{ij}(z)$ of minimal absolute value.
 
 \end{lem}
   \begin{proof}
 By definition,     $p_{ij}(z)=\sum\limits_{k=0}^{k=\infty}(z^kA^k)_{ij}$,  for $z$ with $\abs{z}<\frac{1}{n}$. Since $(I+zA+z^2A^2+...)(I-zA)=I$, 
   $I+zA+z^2A^2+...=(I-zA)^{-1}$ (which gives the form of  $p_{if}(z)$ given in Theorem \ref{theo_genfn_stanley}). Let $v \in \mathbb{C}^d$ be an eigenvector of $A$ with eigenvalue $n e^{\frac{2\pi i(-m)}{h}}$,  for some $ 0 \leq m\leq h-1$. So, on one hand, $(I+zA+z^2A^2+...)(I-zA)\,v=Iv=v$. On the second hand, $(I+zA+z^2A^2+...)(I-zA)\,v=(I+zA+z^2A^2+...)((I-zA)\,v)$, 
   and if $z \rightarrow \frac{1}{n} e^{\frac{2\pi im}{h}}$, then 
  $(I-zA)\,v \rightarrow (I-\frac{1}{n} e^{\frac{2\pi im}{h}}A)\,v\rightarrow 0$. By definition, $v\neq \vec{0}$, so  if,  whenever $z \rightarrow \frac{1}{n} e^{\frac{2\pi im}{h}}$, all the elements in the matrix  $(I-zA)^{-1}$ are finite,  then we get a contradiction. So, there are  $1 \leq k,l \leq d$, such that $(I-zA)^{-1}_{kl}\rightarrow \infty$, whenever $z \rightarrow \frac{1}{n} e^{\frac{2\pi im}{h}}$, that is $\{\frac{1}{n} e^{\frac{2\pi im}{h}}\mid 0 \leq m\leq h-1\}$ is a set of  poles of $p_{kl}(z)$. So, from Lemma \ref{lem_asympto},  $\{\frac{1}{n} e^{\frac{2\pi im}{h}}\mid 0 \leq m\leq h-1\}$ is a set of  poles of $p_{ij}(z)$, for every $1 \leq i,j \leq d$,  and these are 
   simple poles of minimal absolute value from Lemma \ref{lem_1-nz}.
    \end{proof}

\subsection{Proof of the main result: proofs of Theorem 1 and 2}

Let $F_n$ be the free group on $n \geq 1$ generators. Let $\{H_i\alpha_i\}_{i=1}^{i=s}$ be a coset  partition of $F_n$ with $H_i<F_n$ of index $d_i$, $\alpha_i \in F_n$, $1 \leq i \leq s$, and $1<d_1 \leq ...\leq d_s$.  Let $\tilde{X}_{i}$ denote the  Schreier  graph of $H_i$, with  transition matrix   $A_i$ of period $h_i\geq 1$, $1 \leq i\leq s$.  Let  $\tilde{X}_{H_i\alpha_i}$ denote the  Schreier  automaton  of $H_i\alpha_i$, with generating function $p_i(z)$, $1 \leq i\leq s$. We prove some properties of  the generating functions. 
\begin{lem}\label{lem_partition_genfn=sum}
Let $\abs{z}<\frac{1}{n}$. Then
\begin{enumerate}[(i)]
\item{\begin{equation}\label{eqn_partition}
\sum\limits_{i=1}^{i=s}p_{i}(z)= \frac{1}{1-nz}
\end{equation} }
\item For every $1 \leq i \leq s$, $\{\frac{1}{n} e^{\frac{2\pi im}{h_i}}\mid 0 \leq m\leq h_i-1\}$ is a set of simple poles of $p_{i}(z)$ of minimal absolute value.
\end{enumerate}
\end{lem}
\begin{proof}
$(i)$ The  generating  function of $F_n$ is $\frac{1}{1-nz}$,  for $\abs{z} <\frac{1}{n}$. As $\{H_i\alpha_i\}_{i=1}^{i=s}$ is  a coset  partition of $F_n, $  the generating  function of $F_n$ is equal to the sum of the corresponding generating functions.$(ii)$ results from  Lemma \ref{lem_set-poles-simple}.
\end{proof}

\begin{lem}\label{lem_max-period-repeats}
Let $h>1$, where  $h=max\{h_i \mid 1 \leq i \leq s\}$.  Assume $h_k=h$. Let $J=\{j \,\mid\, 1 \leq j\leq s,\,h_j=h\}$. Then
   \begin{enumerate}[(i)]
\item  There exists (at least one) $j\neq k$ such that $\frac{1}{n}e^{\frac{2\pi i}{h}}$ is also a pole of $p_j(z)$ and  $h_j=h$.
 \item  $\sum\limits_{j\in J}Res(p_{j}(z), \frac{1}{n}e^{\frac{2\pi i}{h}})=0$ and moreover $\sum\limits_{j\in J}Res(p_{j}(z), \frac{1}{n}e^{\frac{2\pi im}{h}})=0$, for every $m$ with $gcd(m,h)=1$ .
\end{enumerate}
\end{lem}
\begin{proof}
$(i)$ From Lemma \ref{lem_set-poles-simple}, 
$\{\frac{1}{n}e^{\frac{2\pi im}{h}}\mid 0 \leq m\leq h-1\}$ is a set of simple poles of $p_k(z)$.
Let $z \rightarrow\frac{1}{n}e^{\frac{2\pi i}{h}}$ in Eqn. \ref{eqn_partition}. Then $p_k(z)\rightarrow\infty$ and the left-hand side of Eqn. \ref{eqn_partition} also, while the right-hand side of Eqn. \ref{eqn_partition} is a finite number, a contradiction.   So, there exists  $j\neq k$ such that 
 $\frac{1}{n}e^{\frac{2\pi i}{h}}$ is  a simple pole of $p_{j}(z)$ which implies that $h_j=h$. \\
$(ii)$ From Lemma \ref{lem_partition_genfn=sum}$(i)$,    $\sum\limits_{i=1}^{i=s}Res(p_{i}(z), \frac{1}{n}e^{\frac{2\pi i}{h}})=Res(\frac{1}{1-nz},\frac{1}{n}e^{\frac{2\pi i}{h}})=0$. For every $i \notin J$, $Res(p_{i}(z), \frac{1}{n}e^{\frac{2\pi i}{h}})=0$, since $\frac{1}{n}e^{\frac{2\pi i}{h}}$ is not a pole, so $\sum\limits_{j\in J}Res(p_{j}(z), \frac{1}{n}e^{\frac{2\pi i}{h}})=0$. Clearly, $\sum\limits_{j\in J}Res(p_{j}(z), \frac{1}{n}e^{\frac{2\pi im}{h}})=0$, for every $m$ with $gcd(m,h)=1$.
\end{proof}
Note that using exactly the same argument as in the proof of Lemma \ref{lem_max-period-repeats}, we can prove the following.
\begin{lem}\label{lem_pole-repeats} If $\frac{1}{\lambda}$ is any pole of order $n_{\lambda}$ of  $p_k(z)$, for some  $1 \leq k \leq s$. Then there exists (at least one) $j\neq k$ such that $\frac{1}{\lambda}$ is a pole of the same order $n_{\lambda}$ of  $p_j(z)$ and for $J_\lambda=\{j \,\mid\, 1 \leq j\leq s,\,\frac{1}{\lambda}$ is a pole of $p_j(z)\, \}$,  $\sum\limits_{j\in J_\lambda}Res(p_{j}(z), \frac{1}{\lambda})=0$.
\end{lem} 
Note that although, whenever $\abs{\lambda}<n$, $\frac{1}{\lambda}$ is outside the closure of the domain of convergence of the power series to the generating functions, Lemma \ref{lem_pole-repeats} may be useful. Indeed, in case $\abs{J}=2$, with $J=\{j,k\}$, $p_j(z)$ and $p_k(z)$ have necessarily the same poles, that is  the transition matrices  $A_j$ and $A_k$ have the same non-zero eigenvalues. So, $d_j-n_{j0}=d_k-n_{k0}$, where $n_{j0}$ and $n_{k0}$ denote the algebraic multiplicity of $0$ as an eigenvalue of $A_j$ and $A_k$ respectively. So,  if $A_j$ and $A_k$ are both invertible ($n_{j0}=n_{k0}=0$)
or both not invertible ($n_{j0}=n_{k0}=h$), then $d_j=d_k$.

 \begin{proof}[Proof of Theorem \ref{theo0}]
$(i)$ The proof appears in Lemma \ref{lem_max-period-repeats}.\\
 $(ii)$ If $h=h_{\ell}>1$ does not properly divide any other period $h_i$, $1 \leq i \leq s$, then $e^{\frac{2\pi i}{h}}$ is not  equal to   $e^{\frac{2\pi im}{h'}}$,  for any $m>1$ and any $h' \in \{h_i \mid 1 \leq i \leq s\}$. Using  exactly the same  argument as in  the proof of Lemma \ref{lem_max-period-repeats}, there exists  $j\neq \ell$ such that  $\frac{1}{n}e^{\frac{2\pi i}{h}}$ is a pole of $p_{j}(z)$ and $h_j=h$. \\
 $(iii)$ We prove that for every $h_k$, there exists  $j\neq k$ such that  either $h_j=h_k$ or $h_k\mid h_j$. If $h_k=1$, then clearly $h_k$  divides every period $h_i$, $1 \leq i \leq s$. Assume $h_k>1$ and assume by contradiction that $h_k$  does not divide any period $h_i$, $1 \leq i \leq s,\, i\neq k$. That is, for every $m\geq 1$,  $e^{\frac{2\pi i}{(\frac{h_i}{m})}}$ is not  equal to  $e^{\frac{2\pi im}{h_k}}$,  for every  $h_i$, $1 \leq i \leq s, \, i\neq k$. Using  the same  argument as in  the proof of Lemma \ref{lem_max-period-repeats}, with $z \rightarrow\frac{1}{n}e^{\frac{2\pi i}{h_k}}$ in Eqn. \ref{eqn_partition}, we get a contradiction. So, either there exists  $j\neq k$ such that  $h_j=h_k$ or $h_k\mid h_j$.
\end{proof}
\begin{proof}[Proof of Theorem \ref{theo1}]
$(i)$ If the period of $A_s$ is $d_s$, then $d_s$ is maximal amongst all the periods, and from Lemma \ref{lem_max-period-repeats}, there is a repetition of the maximal period, so  there exists  $j\neq s$ such that $h_j=d_s$. As $h_j\leq d_j$,  $d_j=d_s$. \\ 
$(ii)$ If the period of $A_k$ is $d_k$, then there exists $j \in J$ such that $h_j=d_k$. As $h_j\leq d_j\leq  d_k$, $d_j=d_k$.  
\end{proof}
Let $\{H_i\alpha_i\}_{i=1}^{i=s}$ be a coset  partition of $F_n$ with $H_i<F_n$ of index $d_i$, $\alpha_i \in F_n$, $1 \leq i \leq s$, and $1<d_1 \leq ...\leq d_s$.
A question that arises naturally is whether there always exists a set of $n$ generators of $F_n$ such that  there is a subgroup $H_k$ with Schreier  graph  $\tilde{X}_{k}$, such that its transition matrix $A_k$ has  period $h_k$ greater than $1$. 
A transition matrix $A_i$ is aperiodic if and only if the gcd of the length all the closed (directed) loops in $\tilde{X}_{i}$ is $1$. Every  $\tilde{X}_{i}$ is an Eulerian (directed) loop of length $nd_i$, since each vertex has its in-degree equal to its out-degree and both equal to $n$. In a trial and error approach, it is possible to take consecutive initial paths in  the Eulerian loop, such that their labelling gives a new set of generators and such that no loop of length $1$ occurs anymore. It remains to check that,  with this new labelling of edges, the gcd of the lengths all the closed (directed) loops  is not $1$. It would be interesting to construct a rigorous  algorithm that provides  a such a set of generators.

\bigskip\bigskip\noindent
{ Fabienne Chouraqui,}

\smallskip\noindent
University of Haifa at Oranim, Israel.
\smallskip\noindent
E-mail: {\tt fabienne.chouraqui@gmail.com} {\tt fchoura@sci.haifa.ac.il}


\begin{thebibliography}{40}
\bibitem{bellman}R. Bellman, \emph{Matrix Analysis},  S.I.A.M. Press, 1997. 
\bibitem{berger1}M.A. Berger, A. Felzenbaum, A.S. Fraenkel, \emph{Improvements to two results concerning systems of residue sets}, Ars. Combin. {\bf 20} (1985), 69-82.
\bibitem{berger2}M.A. Berger, A. Felzenbaum, A.S. Fraenkel, \emph{The Herzog-Sch\"onheim conjecture for finite nilpotent groups}, Canad. Math. Bull. {\bf 29}(1986),329-333.
\bibitem{berger3}M.A. Berger, A. Felzenbaum, A.S. Fraenkel, \emph{Lattice parallelotopes and disjoint covering systems}, Discrete Math. {\bf 65} (1987), 23-44.
\bibitem{berger4}M.A. Berger, A. Felzenbaum, A.S. Fraenkel, \emph{Remark on the multiplicity of a partition of a group into cosets}, Fund. Math. {\bf 128} (1987), 139-144.
\bibitem{brodie} M.A. Brodie, R.F. Chamberlain, L.C Kappe, \emph{Finite coverings by normal subgroups}, Proc. Amer. Math. Soc. {\bf 104} (1988), 669-674.
\bibitem{chou_hs} F. Chouraqui, \emph{The Herzog-Sch\"onheim conjecture for finitely generated groups}, ArXiv 1803.08301.
\bibitem{chou-space}F. Chouraqui, \emph{The space of coset partitions of $F_n$ and  Herzog-Sch\"onheim conjecture}, ArXiv 1804.11103.
\bibitem{epstein}D.B.A. Epstein, J.W. Cannon, D.F. Holt, S.V.F. Levy, M.S. Paterson, W.P. Thurston, \emph{Word Processing in Groups}, Jones and Bartlett Publishers (1992).
\bibitem{epstein-zwik}D.B.A. Epstein, A.R Iano-Fletcher, U.Zwick, \emph{Growth functions and automatic groups}, Experimental Math. {\bf 5} (1996), n.4.
\bibitem{erdos} P. Erd\H{o}s, \emph{Egy kongruenciarenslszerekr\H{o}l sz\'ol\'o probl\'em\'ar\'ol}, Matematikai  Lapok, {\bf 4} (1952),
122-128. 
\bibitem{erdos1}P. Erd\H{o}s, \emph{On integers of the form $2^k+p$ and some related problems}, Summa Brasil. Math. {\bf 2} (1950), 113-123.
\bibitem{erdos2}P. Erd\H{o}s, \emph{Problems and results in Number theory}, Recent Progress in Analytic Number Theory, vol. 1, Academic Press, London-New York,  1981, 1-13.
\bibitem{flajolet} P. Flajolet, R.Sedgewick, \emph{Analytic Combinatorics}, Cambridge University press 2009.

\bibitem{geogh}R. Geoghegan, \emph{Topological Methods in Group Theory}, Graduate Texts in Mathematics {\bf 243}, Springer-Verlag,  Berlin, Heidelberg, New York (2008).
\bibitem{ginosar}, Y. Ginosar, \emph{Tile the group}, Elem. Math. {\bf 72}, Swiss Math. Society, to appear.
\bibitem{gino-of} Y. Ginosar, O. Schnabel, \emph{Prime factorization conditions providing multiplicities in coset partitions of groups}, J. Comb. Number Theory, { \bf 3 } (2011), n.2, 75-86.

\bibitem{herzog} M. Herzog, J. Sch\"onheim, \emph{Research problem no. 9}, Canad. Math. Bull., { \bf 17 } (1974), 150.

\bibitem{korec} I. Korec, $\breve{S}$. Zn$\acute{a}$m, \emph{On disjoint covering of groups by their cosets}, Math. Slovaca, { \bf 27 } (1977), 3-7.
\bibitem{lam} T. Lam, K. Leung, \emph{On vanishing sums of roots of unity}, J. Algebra { \bf 224 } (2000), n.1, 91-109.
\bibitem{schnabel}L. Margolis, O. Schnabel, \emph{The Herzog-Sch\"onheim conjecture for small   groups and harmonic subgroups} , ArXiv 1803.03569.
\bibitem{newman} M. Newman, \emph{Roots of unity and covering sets}, Math. Ann. { \bf 191 } (1971), 279-282.
\bibitem{znam}B. Nov$\acute{a}$k, $\breve{S}$. Zn$\acute{a}$m, \emph{Disjoint covering systems}, Amer. Math. Monthly, { \bf 81 } (1974), 42-45.
 \bibitem{pin}J.E. Pin, \emph{On reversible automata}, Lecture Notes in Computer Science {\bf 583}, Springer 1992, p. 401-416.
 \bibitem{pin2}J.E. Pin, \emph{Mathematical foundations of automata theory}, https://www.irif.fr/~jep/PDF/MPRI/MPRI.pdf
 \bibitem{por1}$\breve{S}$. Porubsk$\acute{y}$, \emph{Natural exactly covering systems of congruences}, Czechoslovak Math. J. { \bf 24 } (1974), 598-606.
 \bibitem{por2}$\breve{S}$. Porubsk$\acute{y}$, \emph{Covering systems and generating functions}, Acta Arith. { \bf 26 } (1975), n.3, 223-231.
 \bibitem{por3}$\breve{S}$. Porubsk$\acute{y}$, \emph{Results and problems on covering systems of residue classes}, Mitt. Math. Sem. Giessen,  { \bf 150 } (1981).
 \bibitem{por4}$\breve{S}$. Porubsk$\acute{y}$, J. Sch\"onheim, \emph{Covering systems of Paul Erd\H{o}s. Past, present and future. Paul Erd\H{o}s and his mathematics}, J$\acute{a}$nos Bolyai Math. Soc., { \bf 11 }  (2002), 581-627.

\bibitem{sims} C.C. Sims, \emph{Computation with finitely presented groups}, Encyclopedia of Mathematics and its Applications { \bf 48 },  Cambridge University Press (1994).

\bibitem{stanley}R.P. Stanley, \emph{Enumerative Combinatorics}, Wadsworth and Brooks/Cole, Monterey, CA, 1986.
\bibitem{stilwell}J. Stillwell, \emph{Classical Topology and Combinatorial Group Theory}, Graduate Texts in Mathematics {\bf 72}, Springer-Verlag,  Berlin, Heidelberg, New York (1980).

\bibitem{sun} Z.W. Sun, \emph{Finite covers of groups by cosets or subgroups}, Internat. J. Math. { \bf 17 } (2006), n.9, 1047-1064.
\bibitem{sun2} Z.W. Sun, \emph{An improvement of the Zn$\acute{a}$m-Newman result}, Chinese Quart. J. Math. { \bf 6} (1991), n.3, 90-96.
\bibitem{sun3} Z.W. Sun, \emph{Covering  the integers by arithmetic sequences II} Trans. Amer. Math. Soc. { \bf 348} (1996),4279-4320.
\bibitem{sun-site} Z.W. Sun, \emph{Classified publications on covering systems}, $http://math.nju.edu.cn/~zwsun/Cref.pdf$.
\bibitem{tomkinson1} M.J. Tomkinson, \emph{Groups covered by abelian subgroups}, London Math. Soc. Lecture Note Ser. {\bf 121},  Cambridge Univ. Press (1986).
\bibitem{tomkinson2} M.J. Tomkinson, \emph{Groups covered by finitely many cosets or subgroups}, Comm. Algebra {\bf 15}(1987), 845-859. 
\bibitem{znam} $\breve{S}$. Zn$\acute{a}$m, \emph{On Exactly Covering Systems of Arithmetic Sequences}, Math. Ann. {\bf 180} (1969), 227-232. 
  \end{thebibliography}
\end{document}